\documentclass[a4paper,reqno,9pt]{amsart}

\usepackage{amsmath}
\usepackage{amsfonts}
\usepackage{palatino}
\usepackage{a4wide}
\usepackage{amssymb}
\usepackage{amsthm}
\usepackage{stmaryrd}

\usepackage{epsfig}

\usepackage[usenames]{color}
\definecolor{citegreen}{rgb}{0,0.6,0}
\definecolor{refred}{rgb}{0.8,0,0}

\usepackage[]{draftcopy}
\draftcopyName{ Draft}{240}
\draftcopySetScale{270}

\theoremstyle{plain}

\newtheorem{teo}{Theorem}[section]
\newtheorem{lemma}[teo]{Lemma}
\newtheorem{prop}[teo]{Proposition}

\newtheorem{ackn}{Acknowledgments\!}

\theoremstyle{definition}

\theoremstyle{remark}

\numberwithin{equation}{section}

\def\tr{\operatornamewithlimits{tr}\nolimits}
\def\R{{{\mathbb R}}}

\def\RRR{{\mathrm R}}

\def\Ric{{\mathrm {Ric}}}
\def\Riem{{\mathrm {Riem}}}
\def\Scal{{\RRR}}

\def\KKK{\mathrm {K}}

\begin{document}
\title[Short--Time Existence of the Second Order RG--Flow in Dimension
Three]{Short--Time Existence of the Second Order Renormalization
  Group Flow in Dimension Three}

\author[Laura Cremaschi]{Laura Cremaschi}
\address[Laura Cremaschi]{Scuola Normale Superiore, Piazza dei Cavalieri 7, Pisa, Italy, 56126}
\email[L. Cremaschi]{laura.cremaschi@sns.it}

\author[Carlo Mantegazza]{Carlo Mantegazza}
\address[Carlo Mantegazza]{Scuola Normale Superiore, Piazza dei Cavalieri 7, Pisa, Italy, 56126}
\email[C. Mantegazza]{c.mantegazza@sns.it}

\date{\today}

\begin{abstract} Given a compact three--manifold together with a
  Riemannian metric, we prove the short--time existence of a solution
  to the renormalization group flow, truncated at the second order
  term, under a suitable hypothesis on the sectional curvature of the
  initial metric.
\end{abstract}

\maketitle

\tableofcontents

\section{Introduction}

The \emph{renormalization group} (RG) arises in modern theoretical
physics as a method to investigate the changes of a system
viewed at different distance scales. Since its introduction in the
early~'50, this set of ideas has given rise to significant
developments in \emph{quantum field theory} (QFT) and opened 
connections between contemporary physics and Riemannian geometry. In
spite of this, the RG still lacks of a strong mathematical foundation.\\
In this paper we deal with a particular example in string theory,
the flow equation for the world--sheet nonlinear sigma--models, and
we try to analyze the contribution given from its second order
truncation.
More precisely, Let $S$ be the classical (harmonic map) action
\begin{equation*}
S(\varphi)=\frac{1}{4\pi a}\int_{\Sigma}\tr_h(\varphi^*g)\,d\mu_h
\end{equation*}
where $\varphi:\Sigma\to M^n$ is a smooth map between a surface $(\Sigma, h)$ and
a Riemannian manifold $(M^n,g)$ of dimension $n\geq 3$. The quantity $a>0$ 
is the so--called \emph{string coupling constant}. 
Roughly speaking, in order to control  the path integral quantization of
the action $S$, one introduces a cut--off momentum $\Lambda$ which
parametrizes the spectrum of fluctuations of the theory as the distance
scale is varied as $1/\Lambda \rightarrow 1/\Lambda^\prime$. This formally
generates a flow (the renormalization group flow) in the space of actions 
which is controlled by the induced scale--dependence in $(M^n, g)$. 
Setting $\tau:=-\ln(\Lambda/\Lambda^\prime)$, one thus considers the
so--called {\em beta functions} $\beta$, associated with the
renormalization group of the theory and defined by the formal flow
$g(\tau)$ satisfying
\begin{equation*}
\frac{\partial g_{ik}(\tau)}{\partial \tau}= - \beta_{ik}.
\end{equation*}
In the perturbative regime (that is, when $a\vert\Riem(g)\vert\ll 1$)
the beta functions $\beta_{ik}$ can be expanded in powers of $a$, with 
coefficients which are polynomial in the curvature tensor of the 
metric $g$ and its derivatives. As the quantity $a\vert\Riem(g)\vert$ is 
supposed to be very small, the first order truncation should provide a good
approximation of the full RG--flow
\begin{equation*}
\frac{\partial g_{ik}}{\partial \tau}=-a\RRR_{ik} + o(a)\,,
\end{equation*}
as $a\to 0$.\\
Hence, the first order truncation (with the substitution $\tau=t/2a$) coincides with the
\emph{Ricci flow} $\partial_t g=-2\Ric$, as noted by 
Friedan~\cite{friedan1,friedan2} and Lott~\cite{lott3}, see also~\cite{carfora2}.\\
It is a well--known fact that generally the Ricci flow
becomes singular in finite time and in~\cite{hamilton1} Hamilton
proved that at a finite singular time $T>0$,
the Riemann curvature blows up. Then, near a singularity, 
the Ricci flow is no longer a valid approximation of the behavior
of the sigma--model. From the physical point of view, it appears
then relevant to possibly consider the coupled flow generated by a more general
action, as in~\cite{carfora1,olisunwoo}.

Another possibility could be to consider also the second
order term in the expansion of the beta functions, whose coefficients
are quadratic in the curvature and therefore are (possibly)
dominating, even when $a\vert\Riem(g)\vert\to0$. The resulting flow
is called {\em two--loop} RG--flow
\begin{equation}\label{RGeqorig}
\frac{\partial g_{ik}}{\partial \tau}=-a\RRR_{ik}
-\frac{a^2}{2}\RRR_{ijlm}\RRR_{kstu}g^{js}g^{lt}g^{mu}\,,
\end{equation}
see~\cite{jajomo}. We refer to it as RG$^{2,a}$--flow.

In~\cite{oliynyk1} Oliynyk investigates the behavior of such flow in
dimension two, proving that it can differ substantially from the Ricci
flow. In~\cite{gunoliynik} Guenther and Oliynyk prove the existence and
the stability of the two--loop RG--flow on the $n$--dimensional torus, while in negative constant curvature they prove stability for a modified RG--flow by diffeomorphism and scaling. In~\cite{gigunisen} Gimre,
Guenther and Isenberg study the flow on $n$--dimensional compact
manifolds with constant sectional curvature, observing that in negative
curvature, the asymptotic behavior of the flow depends on the value
of the coupling constant $a$ and of the sectional curvature. In the same
paper, the authors also focus on three--dimensional locally homogeneous
spaces, where the strong assumptions
on the geometry of the initial metric allow to reduce the PDE
to a system of ODEs.\\

The curvature tensor of a Riemannian manifold $(M^n,g)$ is
defined, as in~\cite{gahula}, by
$$
\Riem(X,Y)Z=\nabla_{Y}\nabla_{X}Z-\nabla_{X}\nabla_{Y}Z+\nabla_{[X,Y]}Z\,,
$$ 
while the associated $(4,0)$--tensor is defined by $\Riem(X,Y,Z,T)=g(\Riem(X,Y)Z,T)$.
In local coordinates, we have 
$$
\RRR_{ijkl}=g\Big(\Riem\Big(\frac{\partial}{\partial
  x^i},\frac{\partial}{\partial x^j}\Big)
\frac{\partial}{\partial x^k},\frac{\partial}{\partial x^l}\Big)\,,
$$
the Ricci tensor is then obtained by tracing $\RRR_{ik}=g^{jl}\RRR_{ijkl}$.\\ 
The sectional curvature of a plane $\pi\in T_pM^n$ spanned by a pair of vectors $X,Y\in
T_pM^n$ is defined as
$$
\KKK(\pi)=\KKK(X,Y)=\frac{\Riem(X,Y,X,Y)}{g(X,X)g(Y,Y)-g(X,Y)^2}.
$$

After rescaling the flow parameter $\tau\to t/2a$ in equation~\eqref{RGeqorig}, 
the RG$^{2,a}$--flow is given by
\begin{equation*}
\partial_t g_{ik}=-2\RRR_{ik}-a\RRR_{ijlm}\RRR_{kstu}g^{js}g^{lt}g^{mu}\,,
\end{equation*}
which can be seen as a sort of ``perturbation'' of the Ricci
flow $\partial_t g_{ik}=-2\RRR_{ik}$.

In the paper, we are going to consider the short--time existence
of this flow for an initial three--dimensional, smooth, compact Riemannian
manifold.\\
In the special three--dimensional case, thanks to the algebraic decomposition of the Riemann tensor, the evolution equation has the following expression.
\begin{equation*}
\partial_t g_{ik}=-2\RRR_{ik}-a(2\RRR\RRR_{ik}-2\RRR_{ik}^2+2|\Ric|^2g_{ik}-\RRR^2g_{ik})\,,
\end{equation*}
where $\RRR^2_{ik}=\RRR_{ij}\RRR_{lk}g^{jl}$.

\begin{teo}\label{ST}
Let $(M^3, g_0)$ be a compact, smooth, three--dimensional Riemannian
manifold and $a\in\R$. Assume that the sectional curvature $\KKK_0$ of the initial metric $g_0$ satisfies
\begin{equation}\label{cond}
1+2a\KKK_0(X,Y)>0
\end{equation}
for every point $p\in M^3$ and vectors $X,Y\in T_pM^3$. Let
\begin{equation*}
Lg_{ik}=-2\RRR_{ik}-a\RRR_{ijlm}\RRR_{kstu}g^{js}g^{lt}g^{mu}\,,
\end{equation*}
then, there exists some $T>0$ such that the Cauchy problem
\begin{equation}\label{RG}
\displaystyle
\left\{ 
\begin{array}{ll}
\partial_t g=Lg\\
g(0)=g_0
\end{array}
\right.
\end{equation}
admits a unique smooth solution $g(t)$ for $t\in [0,T)$.
\end{teo}

Notice that, even if not physically relevant, in this theorem we also allow $a<0$. In such case the condition on the initial metric becomes
\begin{equation*}
\KKK_0(X,Y)<-\frac{1}{2a}
\end{equation*}
which is clearly satisfied by every manifold with negative curvature.\\
Any manifold with positive curvature satisfies instead condition~\eqref{cond}, for every $a>0$.

\medskip

\begin{ackn} We are indebted with Mauro Carfora for several valuable
  suggestions.\\
The authors were partially supported by
  the Italian FIRB Ideas ``Analysis and Beyond''.
\end{ackn}

\section{The Principal Symbol of the Operator $L$}
\label{symbsec}

The evolution problem involves a fully nonlinear second--order differential operator 
$Lg$, which, as for the Ricci flow, can only be weakly elliptic, due
to the invariance of the curvature tensors by the action of the
group of diffeomorphisms of the manifold $M^n$. Hamilton in~\cite{hamilton1} proved the short--time
existence of solutions of the Ricci flow using the Nash--Moser
implicit function theorem, showing that the flow satisfies a
certain first--order integrability condition, namely the contracted
second Bianchi identity. In the present paper we
establish the short--time existence using the so called {\em DeTurck's trick} in~\cite{deturck,deturck2},
following the line of Buckland that in~\cite{buckland} showed the
short--time existence of solutions of the {\em cross curvature flow} (see~\cite{hamchow}) in dimension three, via the same method.\\
From the general existence theory of nonlinear parabolic PDEs
(see~\cite[Chapter~4, Section~4]{aubin0},~\cite{friedman}
or~\cite{mantmart1}, for instance) it follows that the evolution 
equation $\frac{\partial}{\partial t}g=Lg$ 
admits a unique smooth solution for short time if the linearized
operator around the initial data 
$DL_g(h)=\left.\frac{d}{ds}\right\vert_{s=0}L(g+sh)$ is strongly elliptic, that
is, if its principal symbol $\sigma_{\xi}(DL_g)$ has all the eigenvalues with
uniformly positive real parts for any cotangent vector $\xi\neq 0$. We will see that, under the hypotheses of Theorem~\ref{ST}, 
the principal symbol of this linearized operator is nonnegative definite and, 
even if it always contains some zero eigenvalues, 
such zero eigenvalues come only from the diffeomorphism invariance. 
This will allow us to apply DeTurck's trick 
to the RG$^{2,a}$--flow.

We start computing the linearized operator $DL_g$ of the
operator $L$ at a metric $g$.\\
The Riemann and Ricci tensors have the following linearizations,
see~\cite[Theorem~1.174]{besse} or~\cite{topping1}.
\begin{eqnarray*}
D\Riem_g(h)_{ijkm}&=&\frac{1}{2}\Big(-\nabla_j\nabla_mh_{ik}-\nabla_i\nabla_kh_{jm}
+\nabla_i\nabla_mh_{jk}+\nabla_j\nabla_kh_{im}\Big)+{\mathrm  {LOT}}\\
D\Ric_g(h)_{ik}&=& \frac{1}{2}\Big(-\Delta
h_{ik}-\nabla_i\nabla_k\tr(h)
+\nabla_i\nabla^th_{tk}+\nabla_k\nabla^th_{it}\Big)+{\mathrm {LOT}}
\end{eqnarray*}
where we use the metric $g$ to lower and upper indices and ${\mathrm {LOT}}$ stands for {\em lower order terms}.\\
Then, the linearized of $L$ around $g$, for every $h\in S^2M^n$, is given by 
\begin{eqnarray*}
DL_g(h)_{ik}&=& -2D\Ric_g(h)_{ik}-aD\Riem_g{(h)_i}^{stu}\RRR_{kstu}- a\RRR_{istu}D\Riem_g{(h)_k}^{stu}+{\mathrm {LOT}}\\
&=&  \Delta h_{ik}+\nabla_i\nabla_k\tr(h)-\nabla_i\nabla^th_{tk}-\nabla_k\nabla^th_{it}\\
&&-\frac{a}{2}\RRR_{kstu}\big(\nabla^s\nabla^th_i^u+\nabla_i\nabla^uh^{st}-\nabla^s\nabla^uh_i^t-\nabla_i\nabla^th^{su}\big)\\
&&-\frac{a}{2}\RRR_{istu}\big(\nabla^s\nabla^th_k^u+\nabla_k\nabla^uh^{st}-\nabla^s\nabla^uh_k^t-\nabla_k\nabla^th^{su}\big)+{\mathrm {LOT}}\\
&=& \Delta h_{ik}+\nabla_i\nabla_k\tr(h)-\nabla_i\nabla^th_{kt}-\nabla_k\nabla^th_{it}\\
&& +a\RRR_{kstu}\big(\nabla_i\nabla^th^{su}-\nabla^s\nabla^th_i^u\big) + a\RRR_{istu}\big(\nabla_k\nabla^th^{su}-\nabla^s\nabla^th_k^u\big)+{\mathrm {LOT}}
\end{eqnarray*}
where the last passage follows from the symmetries of the Riemann tensor (interchanging the last two indices makes it change sign).

Now we obtain the principal symbol of the linearized operator in the direction
of an arbitrary cotangent vector $\xi$ by replacing each covariant derivative
$\nabla_{\alpha}$ with the corresponding component $\xi_{\alpha}$, 
\begin{eqnarray*}
\sigma_{\xi}(DL_g)(h)_{ik} &=& \xi^t\xi_th_{ik}+\xi_i\xi_k\tr(h)-\xi_i\xi^th_{kt}-\xi_k\xi^th_{it}\\ \nonumber
&& +a\RRR_{kstu}\big(\xi_i\xi^th^{su}-\xi^s\xi^th_i^u\big)+ a\RRR_{istu}\big(\xi_k\xi^th^{su}-\xi^s\xi^th_k^u\big)\,.\nonumber
\end{eqnarray*}
Since the symbol is homogeneous, we can assume that $|\xi|_g=1$,
furthermore, we can assume to do all the following computations in an 
orthonormal basis $\{e_1,\dots,e_n\}$ of $T_pM^n$ such that
$\xi=g(e_1,\cdot)$, hence, $\xi_i=0$ for $i\neq 1$.\\
Then, we obtain,
\begin{eqnarray*}
\sigma_{\xi}(DL_g)(h)_{ik} 
&=& h_{ik}+\delta_{i1}\delta_{k1}\tr(h)-\delta_{i1}\delta^{t1}h_{tk}-\delta_{k1}\delta^{t1}h_{it}\\
&&+a\RRR_{kstu}\big(\delta_{i1}\delta^{t1}h^{su}-\delta^{s1}\delta^{t1}h_i^u\big)
+a\RRR_{istu}\big(\delta_{k1}\delta^{t1}h^{su}-\delta^{s1}\delta^{t1}h_k^u\big)\\
&=&
h_{ik}+\delta_{i1}\delta_{k1}\tr(h)-\delta_{i1}h_{1k}-\delta_{k1}h_{i1}\\
&&+a\RRR_{ks1u}\delta_{i1}h^{su}-a\RRR_{k11u}h_i^u
+a\RRR_{is1u}\delta_{k1}h^{su}-a\RRR_{i11u}h_k^u\,.
\end{eqnarray*}
So far the dimension $n$ of the manifold was arbitrary, now we carry
out the computation in the special case $n=3$ 
(using again the symmetries of the Riemann tensor), 
\begin{equation*}
\sigma_{\xi}(DL_g)\begin{pmatrix}
h_{11}\\ h_{12}\\ h_{13}\\ h_{22}\\ h_{33}\\ h_{23}\\
\end{pmatrix}
=
\begin{pmatrix}
h_{22}(1+2a\RRR_{1212})+h_{33}(1+2a\RRR_{1313})+h_{23}(4a\RRR_{1213})\\
h_{33}a\RRR_{1323}+h_{23}a\RRR_{1223}\\
h_{22}a\RRR_{1232}+h_{23}a\RRR_{1332}\\
h_{22}(1+2a\RRR_{1212})+h_{23}2a\RRR_{1213}\\
h_{33}(1+2a\RRR_{1313})+h_{23}2a\RRR_{1213}\\
h_{22}a\RRR_{1213}+h_{33}a\RRR_{1213}+h_{23}(1+a\RRR_{1212}+a\RRR_{1313})\\
\end{pmatrix}\,.
\end{equation*}
Then, we conclude that
\begin{equation*}
\sigma_{\xi}(DL_g)=\begin{pmatrix}
0 & 0 & 0 & 1+2a\RRR_{1212} & 1+2a\RRR_{1313} & 4a\RRR_{1213}\\
0 & 0 & 0 & 0 & a\RRR_{1323} & a\RRR_{1223}\\
0 & 0 & 0 & a\RRR_{1232} & 0 & a\RRR_{1332}\\ 
0 & 0 & 0 & 1+2a\RRR_{1212} & 0 & 2a\RRR_{1213}\\
0 & 0 & 0 & 0 & 1+2a\RRR_{1313} & 2a\RRR_{1213}\\
0 & 0 & 0 & a\RRR_{1213} & a\RRR_{1213} & 1+a\RRR_{1212}+a\RRR_{1313}\\
\end{pmatrix}\,.
\end{equation*}
As expected, in the kernel of the principal symbol there is at least the
three--dimensional space of forms 
$h=\xi\otimes\nu+\nu\otimes\xi\in S^2M^3$ 
where $\nu$ is any cotangent vector, that is, the variations of the metric which are tangent to the
orbits of the group of diffeomorphisms (see~\cite[Chapter~3, Section~2]{chknopf} for more details on this).\\
Now we use the algebraic decomposition of the Riemann tensor in
dimension three in order to simplify the computation of the other
eigenvalues.\\
We recall that
\begin{equation*}
\RRR_{ijkl}=(\Ric\varowedge g)_{ijkl}-\frac{\RRR}{4}(g\varowedge g)_{ijkl}
\end{equation*}
where $\RRR$ denotes the scalar curvature, i.e. the trace of the Ricci
tensor, and the symbol $\varowedge$ denotes the Kulkarni--Nomizu product of
two symmetric bilinear forms $p$ and $q$, defined by
\begin{equation*}
(p\varowedge q)(X,Y,Z,T)= p(X,Z)q(Y,T)+p(Y,T)q(X,Z)-p(X,T)q(Y,Z)-p(Y,Z)q(X,T)\,,
\end{equation*}
for every tangent vectors $X,Y,Z,T$.\\
By means of the expression of the Riemann tensor in terms of the Ricci
tensor and since we are in an 
orthonormal basis, the principal symbol can be expressed in the simpler form
\begin{equation*}
\sigma_{\xi}(DL_g)=\begin{pmatrix}
0 & 0 & 0 & 1+a(\RRR-2\RRR_{33}) & 1+a(\RRR-2\RRR_{22}) & 4a\RRR_{23}\\
0 & 0 & 0 & 0 & a\RRR_{12} & -a\RRR_{13}\\
0 & 0 & 0 & a\RRR_{13} & 0 & -a\RRR_{12}\\ 
0 & 0 & 0 & 1+a(\RRR-2\RRR_{33}) & 0 & 2a\RRR_{23}\\
0 & 0 & 0 & 0 & 1+a(\RRR-2\RRR_{22}) & 2a\RRR_{23}\\
0 & 0 & 0 & a\RRR_{23} & a\RRR_{23} & 1+a\RRR_{11}\\
\end{pmatrix}\,.
\end{equation*}

In order to apply the argument of DeTurck in the next section, we need the weak 
ellipticity of the linearized operator. To get that we have to 
compute the eigenvalues of the minor
\begin{equation*}
A=
\begin{pmatrix}
1+a(\RRR-2\RRR_{33}) & 0 & 2a\RRR_{23}\\
0 & 1+a(\RRR-2\RRR_{22}) & 2a\RRR_{23}\\
a\RRR_{23} & a\RRR_{23} & 1+a\RRR_{11}\\
\end{pmatrix}\,.
\end{equation*}

We claim that with a suitable orthonormal change of 
the basis of the plane $span\{e_2,e_2\}=e_1^{\perp}$ we can always
get an orthonormal basis $\{e'_1,e'_2,e'_3\}$ of $T_pM^3$ such that 
$e'_1=e_1$ and $\RRR'_{23}=\Ric(e'_2,e'_3)=0$.\\
Indeed, if $\{e'_2,e'_3\}$ is any orthonormal basis of
$e_1^{\perp}$, we can write
\begin{equation*}
e'_2=\cos\alpha e_2+\sin\alpha e_3 \qquad e'_3=-\sin\alpha e_2+\cos\alpha e_3
\end{equation*}
for some $\alpha\in[0,2\pi)$. Plugging this into the expression of the Ricci tensor, we obtain
\begin{eqnarray*}
\RRR'_{23}&=&\cos\alpha\sin\alpha(\RRR_{33}-\RRR_{22})+(\cos^2\alpha-\sin^2\alpha)\RRR_{23}\\
&=& \frac{1}{2}\sin(2\alpha)(\RRR_{33}-\RRR_{22})+ \cos(2\alpha)\RRR_{23}\,.
\end{eqnarray*}
Hence, in order to have $\RRR'_{23}=0$, it is sufficient to choose 
\begin{equation*} 
\alpha = \left\{ 
\begin{array}{cl} 
\frac{\pi}{4} & \text{ } \text{ if } \RRR_{22}=\RRR_{33},\\ 
\frac{1}{2}\arctan(\frac{2\RRR_{23}}{\RRR_{22}-\RRR_{33}}) & \text{ } \text{ otherwise.}
\end{array} \right.
\end{equation*} 
The matrix written above represents the symbol $\sigma_{\xi}(DL_g)$ 
with respect to a generic orthonormal basis where the first vector
coincides with $g(\xi,\cdot)$, so with this change of basis we obtain
\begin{equation*}
\sigma_{\xi}(DL_g)=\begin{pmatrix}
0 & 0 & 0 & 1+a(\RRR-2\RRR_{33}) & 1+a(\RRR-2\RRR_{22}) & 0\\
0 & 0 & 0 & 0 & a\RRR_{12} & -a\RRR_{13}\\
0 & 0 & 0 & a\RRR_{13} & 0 & -a\RRR_{12}\\ 
0 & 0 & 0 & 1+a(\RRR-2\RRR_{33}) & 0 & 0\\
0 & 0 & 0 & 0 & 1+a(\RRR-2\RRR_{22}) & 0\\
0 & 0 & 0 & 0 & 0 & 1+a\RRR_{11}\\
\end{pmatrix}\,.
\end{equation*}
Hence, the other three eigenvalues of the matrix $\sigma_{\xi}(DL_g)$ are the diagonal elements of the matrix
\begin{equation*}
A=
\begin{pmatrix}
1+a(\RRR-2\RRR_{33}) & 0 & 0\\
0 & 1+a(\RRR-2\RRR_{22}) & 0\\
0 & 0 & 1+a\RRR_{11}\\
\end{pmatrix}\,,
\end{equation*}
that is,
\begin{equation*}
\lambda_1=1+a(\RRR-2\RRR_{33})\,,
 \quad \lambda_2=1+a(\RRR-2\RRR_{22})\,, \quad\lambda_3=1+a\RRR_{11}\,.
\end{equation*}

Now we recall that, if $\{e_j\}_{j=1,\dots, n}$ is an orthonormal
basis of the tangent space, the Ricci quadratic form is the sum of the
sectional curvatures, 
\begin{equation*}
\RRR_{ii}=\sum_{j\neq i}\KKK(e_i,e_j)
\end{equation*}
and the scalar curvature $\Scal$ is given by 
\begin{equation*}
\Scal=\sum_{i=1}^n\RRR_{ii}=2\sum_{i<j}\KKK(e_i,e_j)\,.
\end{equation*}
Then, in dimension three, denoting by $\alpha=\KKK(e_2,e_3)$,
$\beta=\KKK(e_1,e_3)$ and $\gamma=\KKK(e_1,e_2)$, we obtain that the
above eigenvalues are 
\begin{equation*}
\lambda_1=1+2a\gamma\,,\qquad \lambda_2=1+2a\beta\,,\qquad \lambda_3=1+a(\beta+\gamma)\,.
\end{equation*}

It is now easy to see, by the arbitrariness of the cotangent vector $\xi$, that 
these three eigenvalues are positive, hence, the
operator $L$ is weakly elliptic, if and only if all the sectional curvatures $\KKK(X,Y)$ of $(M^3,g)$
satisfy $1+2a\KKK(X,Y)>0$. If this expression is always positive, then 
there are exactly three zero eigenvalues, due to the diffeomorphism invariance of
the operator $L$.

Following the work of DeTurck~\cite{deturck,deturck2} (see also~\cite{bour1}),
we show that Problem~\eqref{RG} is equivalent 
to a Cauchy problem for a strictly 
parabolic operator, modulo the action of the diffeomorphism group of
$M^n$.\\
Given a vector field $V\in TM^n$, we will denote the Lie derivative
along $V$ with $\mathcal{L}_V$.

\begin{prop}[DeTurck's Trick -- Existence part~\cite{deturck,deturck2}]\label{dtr1}
Let $(M^n,g_0)$ be a compact Riemannian manifold. \\
Let $L:S^2M^n\to S^2M^n$ and $V:S^2M^n\to TM^n$ be differential operators such that $L$ is geometric,  
that is, for every smooth diffeomorphism $\psi:M^n\to M^n$ satisfying $L(\psi^*g)=\psi^* (Lg)$.
If the linearized operator $D(L-\mathcal{L}_V)_{g_0}$ is
strongly elliptic, then the problem
\begin{equation*}
\displaystyle
\left\{ 
\begin{array}{ll}
\partial_t g=Lg\\
g(0)=g_0
\end{array}
\right.
\end{equation*}
admits a smooth solution on an open interval $[0,T)$, for some $T>0$.
\end{prop}

The uniqueness part of the statement is more delicate. It uses 
an argument using the existence and
uniqueness of solutions of the {\em harmonic map flow} 
(see~\cite[Chapter~3, Section~4]{chknopf}).

\begin{lemma}\label{propV}
Let $V:S^2M^n\to TM^n $ be ``DeTurck's'' vector field, that is defined by
\begin{equation*}
V^j(g)=-g_0^{jk}g^{pq}\nabla_p\Big(\frac{1}{2}\tr_g(g_0)g_{qk}-(g_0)_{qk}\Big)
=-\frac{1}{2}g_0^{jk}g^{pq}\big(\nabla_k(g_0)_{pq}-\nabla_p(g_0)_{qk}-\nabla_q(g_0)_{pk}\big)\,,
\end{equation*}
where $g_0$ is a Riemannian metric on $M^n$ and $g_0^{jk}$ is the matrix inverse of $g_0$. \\
The following facts hold true.
\begin{itemize}
\item[(i)] The linearization in $g_0$ of the Lie derivative in the direction $V$ is given by
\begin{align*}
(D\mathcal{L}_V)_{g_0}(h)_{ik}=&\,\frac{1}{2}g_0^{pq}\nabla^0_i\{\nabla^0_kh_{pq}-\nabla^0_ph_{qk}-\nabla^0_{q}h_{pk}\}\\
&\,+\frac{1}{2}g_0^{pq}\nabla^0_k\{\nabla^0_ih_{pq}-\nabla^0_ph_{qi}-\nabla^0_{q}h_{pi}\}+{\mathrm {LOT}}\,,
\end{align*}
where $\nabla^0$ is the Levi--Civita connection of the metric
$g_0$. Hence, its principal symbol in the direction $\xi$, with
respect to an orthonormal basis $\{(\xi)^{\flat}, e_2,\dots,e_n\}$, is 
\begin{equation*}
\sigma_{\xi}((D\mathcal{L}_V)_{g_0})=
\left(\begin{array}{c|c|c}
{\text{{$-{\mathrm {Id}_n\,\,\,\,}$}}} & 
\begin{array}{cccc}
1 & 1 & \dots & 1  \\
0 & 0 & \dots & 0 \\
\vdots & \vdots & \ddots & \vdots\\
0 & 0 & \dots & 0
\end{array} & \phantom{aaa}0\phantom{aaa} \\
  \hline
\phantom{aaa} & \phantom{aaa} & \phantom{aaa} \\
0 & 0 & 0\\
\phantom{aaa} & \phantom{aaa} & \phantom{aaa} \\
\hline
\phantom{aaa} & \phantom{aaa} & \phantom{aaa} \\
0 & 0 & 0\\
\phantom{aaa} & \phantom{aaa} & \phantom{aaa}
\end{array}\right)\,,
\end{equation*}
expressed in the coordinates
\begin{equation*}
(h_{11}, h_{12}, \dots, h_{1n}, h_{22}, h_{33},
\dots, h_{nn}, h_{23}, h_{24}, \dots, h_{n-1,n})
\end{equation*}
of any $h\in S^2M^n$.
\item[(ii)] If $\varphi: (M^n,g)\to (M^n,g_0)$ is a diffeomorphism, then $V(\varphi^*g)=-\Delta_{g,g_0}\varphi$, where the harmonic map Laplacian with respect to $g$ and $g_0$ is defined by 
\begin{equation*}
\Delta_{g,g_0}\varphi=\tr_g(\nabla(\varphi_*))
\end{equation*}
with $\nabla$ the connection defined on $T^*M^n\otimes
\varphi^*TM^n$ using the Levi--Civita connections of $g$ and $g_0$
(see~\cite[Chapter~3, Section~4]{chknopf} for more details).
\end{itemize}
\end{lemma}

\begin{proof}[Proof of Theorem~\ref{ST}] 
Our operator $Lg_{ik}=-2\RRR_{ik}-a\RRR_{ijlm}\RRR_{kstu}g^{js}g^{lt}g^{mu}$ is
clearly invariant under diffeomorphisms, hence, in order to show the smooth existence part in 
Theorem~\ref{ST} we only need to check that $D(L-\mathcal{L}_V)_{g_0}$
is strongly elliptic, where $V$ is the vector field defined in
Lemma~\ref{propV}. By the same lemma, with respect to the orthonormal basis $\{e_1,e'_2,e'_3\}$ introduced above, we have 
\begin{equation*}
\sigma_{\xi}(D(L-\mathcal{L}_V)_{g_0})=\begin{pmatrix}
1 & 0 & 0 & a(\RRR-2\RRR_{33}) & a(\RRR-2\RRR_{22}) & 0\\
0 & 1 & 0 & 0 & a\RRR_{12} & -a\RRR_{13}\\
0 & 0 & 1 & a\RRR_{13} & 0 & -a\RRR_{12}\\ 
0 & 0 & 0 & 1+2a\gamma & 0 & 0\\
0 & 0 & 0 & 0 & 1+2a\beta & 0\\
0 & 0 & 0 & 0 & 0 & 1+a(\beta+\gamma)\\
\end{pmatrix}\,.
\end{equation*}
Finally, by the discussion at the end of such section, we conclude that a necessary and sufficient
condition for the strong ellipticity of the linear operator 
$D(L-\mathcal{L}_{V})_{g_0}$ is then that all the sectional curvatures of $(M^3,g_0)$ satisfy 
\begin{equation*}
1+2a\KKK_0(X,Y)>0\,,
\end{equation*}
for every $p\in M^3$ and vectors $X,Y\in T_pM^3$.\\
The uniqueness of the solution can be proven exactly in the same way
as for the Ricci flow. 
Let $g_1(t)$ and $g_2(t)$ be solutions of the RG$^{2,a}$--flow with
the same initial data $g_0$. By parabolicity of the harmonic map flow,
 introduced by Eells and Sampson in~\cite{eelsam}, there exist
 $\varphi_1(t)$ and $\varphi_2(t)$ solutions of 
\begin{equation*}
\displaystyle
\left\{ 
\begin{array}{ll}
\partial_t \varphi_{i}=\Delta_{g_i,g_0}\varphi_i\\
\varphi(0)=Id_{M^3}
\end{array}
\right.
\end{equation*}
Now we define $\widetilde{g}_i=(\varphi_i^{-1})^*g_i$ and, using that
$\frac{d}{dt}\varphi^{-1}=-(\varphi^{-1})_*(\frac{d}{dt}\varphi)$, it
is easy to show that both $\widetilde{g}_1$ and $\widetilde{g}_2$ are
solutions of the Cauchy problem associated to the strong elliptic
operator $L-\mathcal{L}_V$ and starting at the same initial metric
$g_0$, hence they must coincide, by uniqueness. 
By point (ii) of Lemma~\ref{propV}, the
diffeomorphisms $\varphi_i$ also coincide because they are both the
one--parameter group generated by
$-V(\widetilde{g}_1)=-V(\widetilde{g}_2)$. Finally,
$g_1=\varphi_1^*(\widetilde{g}_1)=\varphi_2^*(\widetilde{g}_2)=g_2$
and this concludes the proof of Theorem~\ref{ST}.
\end{proof}

\section{Some Remarks}

In order to continue the study of this flow, some natural questions arise, 
one would like to find some Perelman--type entropy functionals which
are monotone along the flow, as proposed by Tseytlin in~\cite{tseytlin1};
another possibility is to investigate the evolution of the
curvatures along the flow under the hypothesis of Theorem~\ref{ST} 
and try to find (if there are) some preserved conditions in order to explore
the long--time behavior and the structure of the singularities at the maximal
time of existence.

The analysis leading to Theorem~\ref{ST} can be repeated step--by--step for the operator $L_0$, given by
$$
L_0g=-a\RRR_{ijlm}\RRR_{kstu}g^{js}g^{lt}g^{mu}\,,
$$
with associated RG$^{2,a}_0$--flow
\begin{equation*}
\partial_t g_{ik}=-a\RRR_{ijlm}\RRR_{kstu}g^{js}g^{lt}g^{mu}\,.
\end{equation*}
In this case, along the same lines, the existence of a unique smooth evolution of an initial metric  $g_0$ is guaranteed as long as
\begin{equation*}
a\KKK_0(X,Y)>0
\end{equation*}
for every point $p\in M^3$ and vectors $X,Y\in T_pM^3$. 
That is, if $a>0$ when the initial manifold has positive curvature and
if $a<0$ when it has negative curvature.\\
For geometrical purposes, this flow could be more interesting than
the RG$^{2,a}$--flow, in particular because of its scaling invariance, which
is not shared by the latter. 

Another possibility in this direction is given by the {\em squared
  Ricci flow}, that is, the evolution of an initial metric $g_0$
according to
\begin{equation*}
\partial_t g_{ik}=-a\RRR_{ij}\RRR_k^j\,,
\end{equation*}
which is scaling invariant and can be analyzed analogously, or
a ``mixing'' with the Ricci flow (non scaling invariant)
\begin{equation*}
\partial_t g_{ik}=-2\RRR_{ik}-a\RRR_{ij}\RRR_k^j\,,
\end{equation*}
for any constant $a\in\R$, as before.\\
Indeed, the principal symbol of the operator
$H=\RRR_{ij}\RRR_k^j$ can be computed as in
Section~\ref{symbsec}. The linearized of the operator $H$ around a metric $g$, for every $h\in S^2M^n$, is given by 
\begin{eqnarray*}
DH_g(h)_{ik}&=& \RRR_k^jD\Ric_g(h)_{ij} + \RRR_i^jD\Ric_g(h)_{jk} +{\mathrm {LOT}}\\
&=&  \frac{1}{2}  \RRR_k^j         \Big(-\Delta
h_{ij}-\nabla_i\nabla_j\tr(h)
+\nabla_i\nabla^mh_{mj}+\nabla_j\nabla^mh_{im}\Big)\\
&& + \frac{1}{2}\RRR_i^j \Big(-\Delta
h_{jk}-\nabla_j\nabla_k\tr(h)
+\nabla_j\nabla^mh_{mk}+\nabla_k\nabla^mh_{jm}\Big)+{\mathrm {LOT}}\,.
\end{eqnarray*}
Hence, the principal symbol in the direction of the cotangent vector $\xi$, as before, is
\begin{eqnarray*}
\sigma_{\xi}(DH_g)(h)_{ik} &=&  -\frac{1}{2}  \RRR_k^j         
\Big(\xi^m\xi_mh_{ij}+\xi_i\xi_j\tr(h)-\xi_i\xi^mh_{jm}-\xi_j\xi^mh_{im}\Big)\\
&&-\frac{1}{2}  \RRR_i^j        
\Big(\xi^m\xi_mh_{jk}+\xi_j\xi_k\tr(h)-\xi_j\xi^mh_{km}-\xi_k\xi^mh_{jm}\Big)\\
&=&  -\frac{1}{2}  \RRR_k^j         
\Big(h_{ij}+\delta_{1i}\delta_{1j}\tr(h)-\delta_{1i}h_{1j}-\delta_{1j}h_{1i}\Big)\\
&&-\frac{1}{2}  \RRR_i^j
\Big(h_{jk}+\delta_{1j}\delta_{1k}\tr(h)-\delta_{1j}h_{1k}-\delta_{1k}h_{1j}\Big)\\
&=& -\frac{1}{2}\Big(\RRR_{1k}(\delta_{1i}\tr(h)-h_{1i})+\RRR_{1i}(\delta_{1k}\tr(h)-h_{1k})\Big)\\
&& -\frac{1}{2}\Big(\RRR_k^j(h_{ij}-\delta_{1i}h_{1j})+\RRR_i^j(h_{jk}-\delta_{1k}h_{1j}) \Big)\,,
\end{eqnarray*}
where $\xi=g(e_1,\cdot)$ and $\{e_i\}$ is an orthonormal basis of
$T_pM^n$.\\
Again, by specifying the initial metric to be $g_0$ and diagonalizing the restriction of the Ricci tensor (which is still symmetric) to the hyperspace $e_1^{\perp}$, we can find an orthonormal basis $\{e'_2,\dots,e'_n\}$ of
$e_1^{\perp}$ such that $\Ric(e'_i,e'_k)=0$ if $i\not=k$, the principal symbol of the operator 
$H$, computed in the basis $\{e_1, e'_2,\dots,e'_n\}$, is described by 
\begin{align*}
\sigma_{\xi}(DH_{g_0})(h)_{11} &\,= -\frac{1}{2}\Big(2\RRR_{11}\sum_{j=2}^nh_{jj}\Big)&\\
\sigma_{\xi}(DH_{g_0})(h)_{1k} &\,= -\frac{1}{2}
\Big(2\RRR_{1k}h_{kk}+\RRR_{1k}\sum_{j\neq 1,\,k}h_{jj}+\sum_{j\neq
  1,\,k}\RRR_{1j}h_{jk}\Big)\\
\sigma_{\xi}(DH_{g_0})(h)_{kk} &\,= -\frac{1}{2}\Big(2\RRR_{kk}h_{kk}\Big)\\
\sigma_{\xi}(DH_{g_0})(h)_{ik} &\,= -\frac{1}{2}\Big((\RRR_{kk}+\RRR_{ii})h_{ik}\Big)
\end{align*}
for every $i,k\in\{2,\dots,n\}$ with $i\not= k$.\\
It is easy to see that the matrix associated to $\sigma_{\xi}(DH_{g_0})$ expressed in the coordinates
\begin{equation*}
(h_{11}, h_{12}, \dots, h_{1n}, h_{22}, h_{33},
\dots, h_{nn}, h_{23}, h_{24}, \dots, h_{n-1,n})
\end{equation*}
of $S^2M^n$ is upper triangular with $n$ zeroes on the first $n$
diagonal elements, then the next $(n-1)$ ones are the values
$-\RRR_{kk}$ for $k\in\{2,\dots,n\}$ and finally, the last
$(n-1)(n-2)/2$ ones are given by $-(\RRR_{ii}+\RRR_{kk})/2$ for every
$i,k\in\{2,\dots,n\}$ with $i\not= k$.\\
Now, applying Proposition~\ref{dtr1} with the same vector field $V$ of Lemma~\ref{propV}, the squared Ricci flow
\begin{equation*}
\partial_t g_{ik}=-a\RRR_{ij}\RRR_{lk}g^{jl}\,,
\end{equation*}
has a unique smooth solution for short time, when $a>0$ for every
initial manifold $(M^n,g_0)$ with positive Ricci 
curvature and when $a<0$, for every
initial manifold $(M^n,g_0)$ with negative Ricci curvature.

We conclude this discussion mentioning the 
{\em cross curvature flow}, introduced by Chow and Hamilton
in~\cite{hamchow}, which belongs to this ``family'' of quadratic
flows. The short--time existence and uniqueness of a smooth evolution
of every initial metric of a three--dimensional manifold 
with curvature not changing sign, was established by
Buckland in~\cite{buckland}.

\bigskip

\noindent
Note. {\em Recently, Gimre, Guenther and Isenberg extended the short--time existence of the RG$^{2,a}$--flow, Theorem~\ref{ST}, to any dimensions in~\cite{gigunisen3} (see also~\cite{gigunisen2}).}

\bigskip

\bibliographystyle{plain}
\bibliography{biblio}

\begin{thebibliography}{10}

\bibitem{aubin0}
T.~Aubin.
\newblock {\em Some nonlinear problems in {R}iemannian geometry}.
\newblock Springer--Verlag, 1998.

\bibitem{besse}
A.~L. Besse.
\newblock {\em Einstein manifolds}.
\newblock Springer--Verlag, Berlin, 2008.

\bibitem{bour1}
V.~Bour.
\newblock Fourth order curvature flows and geometric applications.
\newblock ArXiv Preprint Server -- http://arxiv.org, 2010.

\bibitem{buckland}
J.~A. Buckland.
\newblock Short--time existence of solutions to the cross curvature flow on
  3--manifolds.
\newblock {\em Proc. Amer. Math. Soc.}, 134(6):1803--1807 (electronic), 2006.

\bibitem{carfora1}
M.~Carfora.
\newblock Renormalization group and the {R}icci flow.
\newblock {\em Milan J. Math.}, 78(1):319--353, 2010.

\bibitem{carfora2}
M.~Carfora and A.~Marzuoli.
\newblock Model geometries in the space of {R}iemannian structures and
  {H}amilton's flow.
\newblock {\em Classical Quantum Gravity}, 5(5):659--693, 1988.

\bibitem{hamchow}
B.~Chow and R.~S. Hamilton.
\newblock The cross curvature flow of 3--manifolds with negative sectional
  curvature.
\newblock {\em Turkish J. Math.}, 28(1):1--10, 2004.

\bibitem{chknopf}
B.~Chow and D.~Knopf.
\newblock {\em The {R}icci flow: an introduction}, volume 110 of {\em
  Mathematical Surveys and Monographs}.
\newblock American Mathematical Society, Providence, RI, 2004.

\bibitem{deturck}
D.~M. DeTurck.
\newblock Deforming metrics in the direction of their {R}icci tensors.
\newblock {\em J. Diff. Geom.}, 18(1):157--162, 1983.

\bibitem{deturck2}
D.~M. DeTurck.
\newblock Deforming metrics in the direction of their {R}icci tensors (improved
  version).
\newblock In H.-D. Cao, B.~Chow, S.-C. Chu, and S.-T. Yau, editors, {\em
  Collected Papers on Ricci Flow}, volume~37 of {\em Series in Geometry and
  Topology}, pages 163--165. Int. Press, 2003.

\bibitem{eelsam}
J.~Jr. Eells and J.~H. Sampson.
\newblock Harmonic mappings of {R}iemannian manifolds.
\newblock {\em Amer. J. Math.}, 86:109--160, 1964.

\bibitem{friedan1}
D.~H. Friedan.
\newblock Nonlinear models in {$2+\varepsilon $} dimensions.
\newblock {\em Phys. Rev. Lett.}, 45(13):1057--1060, 1980.

\bibitem{friedan2}
D.~H. Friedan.
\newblock Nonlinear models in {$2+\varepsilon$} dimensions.
\newblock {\em Ann. Physics}, 163(2):318--419, 1985.

\bibitem{friedman}
A.~Friedman.
\newblock {\em Partial differential equations of parabolic type}.
\newblock Prentice--Hall Inc., Englewood Cliffs, NJ, 1964.

\bibitem{gahula}
S.~Gallot, D.~Hulin, and J.~Lafontaine.
\newblock {\em {R}iemannian Geometry}.
\newblock Springer--Verlag, 1990.

\bibitem{gigunisen2}
K.~Gimre, C.~Guenther, and J.~Isenberg.
\newblock A geometric introduction to the {2--loop} renormalization group flow.
\newblock ArXiv Preprint Server -- http://arxiv.org, 2013.

\bibitem{gigunisen}
K.~Gimre, C.~Guenther, and J.~Isenberg.
\newblock Second--order renormalization group flow of three--dimensional
  homogeneous geometries.
\newblock {\em Comm. Anal. Geom.}, 21(2):435--467, 2013.

\bibitem{gigunisen3}
K.~Gimre, C.~Guenther, and J.~Isenberg.
\newblock Short--time existence for the second order renormalization group flow
  in general dimensions.
\newblock ArXiv Preprint Server -- http://arxiv.org, 2014.

\bibitem{gunoliynik}
C.~Guenther and T.~A. Oliynyk.
\newblock Stability of the (two-loop) renormalization group flow for nonlinear
  sigma models.
\newblock {\em Lett. Math. Phys.}, 84(2-3):149--157, 2008.

\bibitem{hamilton1}
R.~S. Hamilton.
\newblock Three--manifolds with positive {R}icci curvature.
\newblock {\em J. Diff. Geom.}, 17(2):255--306, 1982.

\bibitem{jajomo}
I.~Jack, D.~R.~T. Jones, and N.~Mohammedi.
\newblock A four-loop calculation of the metric {$\beta$}-function for the
  bosonic {$\sigma$}-model and the string effective action.
\newblock {\em Nuclear Phys. B}, 322(2):431--470, 1989.

\bibitem{lott3}
J.~Lott.
\newblock Renormalization group flow for general {$\sigma$}--models.
\newblock {\em Comm. Math. Phys.}, 107(1):165--176, 1986.

\bibitem{mantmart1}
C.~Mantegazza and L.~Martinazzi.
\newblock A note on quasilinear parabolic equations on manifolds.
\newblock {\em Ann. Sc. Norm. Sup. Pisa}, 11 (5):857--874, 2012.

\bibitem{oliynyk1}
T.~A. Oliynyk.
\newblock The second--order renormalization group flow for nonlinear sigma
  models in two dimensions.
\newblock {\em Classical Quantum Gravity}, 26(10):105020, 8, 2009.

\bibitem{olisunwoo}
T.~A. Oliynyk, V.~Suneeta, and E.~Woolgar.
\newblock Metric for gradient renormalization group flow of the worldsheet
  sigma model beyond first order.
\newblock {\em Phys. Rev. D}, 76(4):045001, 7, 2007.

\bibitem{topping1}
P.~Topping.
\newblock {\em Lectures on the {R}icci flow}, volume 325 of {\em London
  Mathematical Society Lecture Note Series}.
\newblock Cambridge University Press, Cambridge, 2006.

\bibitem{tseytlin1}
A.~A. Tseytlin.
\newblock Sigma model renormalization group flow, ``central charge'' action and
  {P}erelman's entropy.
\newblock {\em Phys. Rev. D}, 75(6):064024, 6, 2007.

\end{thebibliography}

\end{document}